\documentclass[10pt,a4paper]{amsart}

\usepackage{amsmath,amssymb,amsthm,graphics,amscd}
\usepackage{cite}
\usepackage[all]{xy}
\long\def\symbolfootnote[#1]#2{\begingroup\def\thefootnote{\fnsymbol{footnote}}
\footnote[#1]{#2}\endgroup}
  \renewenvironment{thebibliography}[1]{
    \begin{oldthebibliography}{#1}
      \setlength{\parskip}{0ex}
      \setlength{\itemsep}{0ex}
  }
  {
    \end{oldthebibliography}
  }
  
\begin{document}

\newcounter{rownum}
\setcounter{rownum}{0}
\newcommand{\ab}{\addtocounter{rownum}{1}\arabic{rownum}}
\newcommand{\im}{\mathrm{im}}
\newcommand{\x}{$\times$}

\newtheorem{lemma}{Lemma}[subsection]
\newtheorem{theorem}[lemma]{Theorem}
\newtheorem{corollary}[lemma]{Corollary}
\newtheorem{conjecture}[lemma]{Conjecture}
\newtheorem{prop}[lemma]{Proposition}
\theoremstyle{remark}
\newtheorem{remark}[lemma]{Remark}
\newtheorem{example}[lemma]{Example}
\theoremstyle{definition}
\newtheorem{defn}[lemma]{Definition}
\newtheorem{question}[lemma]{Question}

\renewcommand{\labelenumi}{(\roman{enumi})}
\newcommand{\Hom}{\mathrm{Hom}}
\newcommand{\Aut}{\mathrm{Aut}}
\newcommand{\Ext}{\mathrm{Ext}}
\newcommand{\Mor}{\mathrm{Mor}}
\newcommand{\soc}{\mathrm{Soc}}
\newcommand{\Lie}{\mathrm{Lie\ }}

\newcommand{\F}{\mathbb{F}}
\newcommand{\G}{\mathbf{G}}

\parindent=0pt
\addtolength{\parskip}{0.5\baselineskip}

\title{On unipotent algebraic $G$-groups and $1$-cohomology}
\author{David I. Stewart}
\address{New College, Oxford, OX1 3BN}
\email{dis20@cantab.net}
\maketitle

{\small {\bf Abstract.}  In this paper we consider non-abelian $1$-cohomology for groups with coefficients in other groups. We prove  versions of the `five lemma' arising from this situation. We go on to show that a %%%%
connected unipotent algebraic group $Q$ acted on morphically by a 
%%%%
connected
%%%%
algebraic group $G$ admits a filtration with successive quotients having the structure of $G$-modules. From these results we deduce extensions to results due to Cline, Parshall, Scott and van der Kallen. First, if $G$ is a connected, reductive algebraic group with Borel subgroup $B$ and $Q$ a unipotent algebraic $G$-group, we show the restriction map $H^1(G,Q)\to H^1(B,Q)$ is an isomorphism. We also show that this situation admits a notion of rational stability and generic cohomology. We use these results to obtain corollaries about complete reducibility and subgroup structure.}
\section{Introduction}
The low degree cohomology of a group $G$ with coefficients in another group $Q$ on which $G$ acts gives rise to important invariants in group theory. For instance in degree $0$, $H^0(G,Q)$ is the group of fixed points of the action of $G$ on $Q$, $H^1(G,Q)$ is the set of conjugacy classes of complements to $Q$ in the semidirect product $GQ$, and $H^2(G,Q)$ (where defined) measures the number of non-equivalent group extensions \[1\to Q\to E\to G\to 1.\]
In this paper we consider the situation where $Q$ is in general, a non-abelian $G$-group. In such a situation $H^1(G,Q)$ is just a pointed set (in contrast to the situation that $Q=V$ is a representation for $G$, when $H^1(G,V)$ is a vector space). One needs to work harder to employ the usual tools of cohomology since exact sequences do not, unfortunately, say very much about pointed sets; for instance, if $1\to A\to B$, is exact, the map $A\to B$ need not be an injection, merely that 
anything mapping to the distinguished element in $B$ %%%
 should have been the distinguished element in $A$. However,%%%%
  one can usually rescue the situation for non-abelian $1$-cohomology by employing `twisting', a technique familiar %%%%
from the theory of Galois cohomology.

Before we describe our results, let us say at the outset, that we will be consider in detail the situation where $G$ is a (possibly reductive) algebraic group over an algebraically closed field $k$, and $Q$ is a unipotent algebraic $G$-group. However up to the end of \S3.1 all our results apply equally in the category of abstract groups. For clarity we give our results up to this point in the abstract setting and then in \S3.2 indicate the appropriate modifications to use these results in the rational (algebraic) situation; the proofs are the same, {\it mutatis mutandis}.

Letting $G$ first be an abstract group and $Q$ an abstract $G$-group, we give in \S\ref{defs}, the well-known definitions of the $0$th, $1$st and $2$nd cohomology groups of $G$ with coefficients in $Q$ in terms of cocycles (where these definitions exist) as well as the notion of twisting an action by a cocycle. When $Q$ fits into a short exact sequence \[1\to R\to Q\to S\to 1\] and where the image of $R$ in $Q$ is contained in the centre $Z(Q)$ of $Q$, there is an exact sequence of non-abelian cohomology (Proposition \ref{list}(i)). This exact sequence fits into various commutative diagrams (Proposition \ref{list}(ii)--(x)). In particular we consider the restriction of this exact sequence to a subgroup $B$ of $G$ and show how conditions on the restriction maps of cohomology $H^n(G,R)\to H^n(B,R)$ and $H^n(G,S)\to H^n(B,S)$ for $0\leq n\leq 2$ can be used to deduce properties of the restriction $H^1(G,Q)\to H^1(B,Q)$ (Theorems \ref{fivelemmaforh1} and \ref{fivelemmaforh1funct}). For abelian $Q$ this would come down to a simple application of the equally simple `five lemma' (see e.g. \cite[2.72]{Rot09}) but in the non-abelian case, we must use the technique of twisting.

We apply this in the special situation where $G$ is %%%%
a connected
%%%%
algebraic group acting (morphically) on a 
%%%%
connected
%%%%
unipotent algebraic group $Q$. We first show that $Q$ has a central filtration by $G$-modules (Theorem \ref{gmodfilt}). This filtration allows us to go further than the results obtained in \cite[\S6]{Ric82} where this general set-up was first considered; in particular, one can bring second cohomology into play.\footnote{We remark that after this paper was submitted it came to our attention that George McNinch has a preprint \cite{McN13} establishing the existence of such a central filtration, while working in the more general context of %%%%
connected algebraic groups over arbitrary fields $k$. In such a situation one can have unipotent radicals of algebraic groups which are not defined over $k$. McNinch defines a condition \textbf{(R)}: the unipotent radical of $G$ is defined over $k$. He then proves that when \textbf{(R)} holds, such a filtration of $Q$ exists whenever $Q$ is split over $k$. Now, since these conditions hold automatically for $k=\bar k$ our result can be recovered from his. Our proofs, however, appear to be substantively different. The one given here is significantly shorter, no doubt owing to the technical difficulties involved in McNinch's more general result. The problem of whether there exists such a filtration of $Q$ when $G$ is disconnected (e.g. finite) remains open---see \cite[Question E]{McN13}.}

Perhaps surprisingly, it turns out that unipotent algebraic groups are just as good as vector spaces as far as low-degree cohomology is concerned. By way of demonstration, we can use our work to generalise several of the results in the paper \cite{CPSV77}. Consider the case where $G$ is connected and reductive; and let $P$ be a parabolic subgroup of $G$. In \cite{CPSV77}, it is proved that $H^n(G,V)\cong H^n(P,V)$ for any finite dimensional rational $G$-module $V$. We use our filtration to extend this result by applying our Theorem \ref{fivelemmaforh1} inductively. We show that $H^1(G,Q)\cong H^1(P,Q)$ for any unipotent algebraic group $Q$ on which $G$ acts morphically.

The main result of \cite{CPSV77} is the proof of the existence of rational stability and generic cohomology. Let $G$ be additionally defined and split over the prime subfield $\mathbb F_p$. If $V^{[r]}$ denotes the $r$th Frobenius twist of $V$ then it shown that for each $m\in\mathbb N$ and $V$ a finite dimensional module there is some $r_0$ with $H^m(G,V^{[r]})\cong H^m(G,V^{[r_0]})$ for any $r\geq r_0$. In addition, there is some $s_0$ such that for each $s\geq s_0$ restriction $H^m(G,V^{[r_0]})\to H^m(G(p^s),V^{[r_0]})\cong H^m(G(p^s),V)$ is an isomorphism. We extend each of these theorems in the case $m=1$, replacing $V$ with a unipotent algebraic group (Theorem \ref{rationalstab}).

The above has a number of corollaries on subgroup structure related to Serre's notion of $G$-complete reducibility. If $H$ is a closed, connected, reductive subgroup of a parabolic subgroup $P$ of $G$, one can find its image $\bar H$ under the map $\pi:P\to L$ for $L$ a Levi subgroup of $P$. One need not have $H$ a complement to $Q$ in $\bar HQ$, even though it is when considered as an abstract group. Where it is  a complement, it corresponds, up to $Q$-conjugacy, to an element of $H^1(\bar H,Q)$; we show in Lemma \ref{cnbn} in almost all cases, $H$ \emph{is} a complement, and even when it is not, it still corresponds to an element of $H^1(\bar H,Q^{[1]})$. This gives a cohomological characterisation of conjugacy classes of subgroups of parabolics. We can then prove results on $G$-complete reducibility using this characterisation, via a result of Bate, Martin, Roerhle and Tange. First, let $B$ be a Borel subgroup of $H$ with unipotent radical $U=R_u(B)$. Then $H$ is conjugate to a subgroup of $L$ if and only if $B$ is, if and only if $U$ is (Corollary \ref{gcr}). Second, for some $q$ which can be explicitly calculated, there is a finite reductive subgroup $H(q)$ of $H$ such that $H$ is conjugate to a subgroup of $L$ if and only if $H(q)$ is (Corollary \ref{finitegcr}).

\section{Preliminaries}
\subsection{Definitions}\label{defs}
We start with some standard material. Let $G$ be a group and let $Q$ be a $G$-set, that is, a set equipped with a right action of $G$.

Then we define $H^0(G,Q)$ as the set $Q^G$ of fixed points of $G$ acting on $Q$, i.e. $H^0(G,Q)=\{q\in Q: q^g=q, \forall g\in G\}$.

If further $Q$ is a $G$-group, then we may define the $1$-cohomology $H^1(G,Q)$ as follows.

First, a $1$-cocycle is a map $\gamma:G\to Q$ satisfying the cocycle condition $\gamma(gh)=\gamma(g)^h\gamma(h)$. The set of all $1$-cocycles is denoted by $Z^1(G,Q)$. We say that two $1$-cocycles $\gamma, \delta$ are cohomologous and write $\gamma\sim\delta$ if $\delta(g)=q^{-g}\gamma(g)q$ for some element $q\in Q$. Then we define $H^1(G,Q)$ as the set of equivalence classes of $Z^1(G,Q)$ under $\sim$; i.e. $H^1(G,Q)=Z^1(G,Q)/\sim$.

There is a distinguished point in $Z^1(G,Q)$ given by the trivial cocycle $\mathbf{1}:G\to Q$; $\mathbf{1}(g)=1$. The equivalence class of the trivial cocycle, $[\mathbf{1}]\subseteq Z^1(G,Q)$ is the set of coboundaries and we denote this distinguished point in $H^1(G,Q)$ by $B^1(G,Q)$. 

If $R$ is an abelian $G$-group we can go further and define the second cohomology: define a $2$-cocycle to be a map $\gamma:G\times G\to R$ satisfying the $2$-cocycle condition $\gamma(gh,k)\gamma(g,h)^k=\gamma(g,hk)\gamma(g,k)$. The set of all $2$-cocycles is denoted $Z^2(G,R)$. We say that two $2$-cocycles $\gamma$ and $\delta$ are cohomologous and write $\gamma\sim\delta$ if there is a map (morphism) $\phi:G\to R$ with $\delta(g,h)=\gamma(g,h)\phi(g)^h\phi(h)\phi(gh)^{-1}$. We then define $H^2(G,R)$ to be the set of equivalence classes of $Z^2(G,R)$ modulo $\sim$, i.e. $H^2(G,R):=Z^2(G,R)/\sim$. Again there is a distinguished point in $Z^2(G,R)$ given by the trivial $2$-cocycle $\mathbf{1}$ and its class in $H^2(G,R)$ is denoted by $B^2(G,R)$.

\subsection{Twisting, sequences and commutative diagrams}
Recall from \cite[I.\S5.3]{Ser94} the notion of twisting the action of $G$ on a $G$-group $Q$ by a cocycle in $Z^1(G,Q)$. To wit, fix an arbitrary $1$-cocycle $\gamma\in Z^1(G,Q)$ and define $q * g=(q^g)^{\gamma(g)}=\gamma(g)^{-1}q^g\gamma(g)$ for each $q\in Q$.\footnote{More generally, $R$ can be a further right $G$-set with a right action on $Q$ which commutes with the action of $G$ on $Q$, i.e. $(q^r)^g=q^{gr^g}$. Then one can define $q*g=q^{g\gamma(g)}$ for $q\in Q$, $g\in G$ and $\gamma\in Z^1(G,R)$.} One checks this is a new action of $G$ on $Q$. We denote the new $G$-group by $Q_\gamma$, which coincides with $Q$ as a group.

We give a few facts about the set-up so far.

\begin{prop}\label{list}
Let $1\to R\to\xymatrix{Q \ar @/^/[r]^\pi\ar@/_/@{<-}[r]_\sigma & S}\to 1$ be a short exact sequence of $G$-groups, with the image of $R$ in $Q$ contained in the centre $Z(Q)$ of $Q$ and with $\sigma$ a section of $\pi$ ($\sigma$ not necessarily a homomorphism). Additionally, fix $\gamma\in Z^1(G,Q)$ and denote the induced element in $Z^1(G,S)$ also by $\gamma$. Then the following hold:
\newcounter{saveenum}
\begin{enumerate}\item There is an exact sequence of cohomology
\begin{align*}1\to H^0(G&,R)\to H^0(G,Q)\to H^0(G,S)\\&\stackrel{\delta_G}{\to} H^1(G,R)\to H^1(G,Q)\to H^1(G,S)\stackrel{\Delta_G}{\to} H^2(G,R),\end{align*}
with $\delta_G(s)(g)=[\sigma(s)^{-g}\sigma(s)]$; $\Delta_G(\gamma)(g,h)=[\sigma(\gamma(g))^h\sigma(\gamma(h))\sigma(\gamma(gh)^{-1})]$.
\item Moreover, the association of the short exact sequence to the longer is functorial: Suppose we have another short exact sequence of $G$-groups \[1\to R'\to\xymatrix{Q' \ar @/^/[r]^{\pi'}\ar@/_/@{<-}[r]_{\sigma'} & S'}\to 1\] with the image of $R'$ in $Z(Q')$ and with $\sigma'$ a section of $\pi'$; and a commutative diagram 
\[\begin{CD} 1@>>> R @>>> Q @>>> S @>>>  1\\
@. @V\zeta VV @V\eta VV @V\theta VV @. \\
1@>>> R' @>>> Q' @>>> S' @>>>  1,\end{CD}\]where the vertical arrows are $G$-equivariant homomorphisms.
Then there is a commutative diagram \begin{align*}\minCDarrowwidth19pt \begin{CD} 1@>>> R^G @>>> Q^G @>>> S^G @>>>  H^1(G,R)\\
@. @V\zeta VV @V\eta VV @V\theta VV @V\zeta\circ\_ VV \\
1@>>> {R'}^G @>>> {Q'}^G @>>> {S'}^G @>>> H^1(G,R') \end{CD}\hspace{80pt}\\
\hspace{80pt}\begin{CD} @>>> H^1(G,Q)@>>> H^1(G,S)@>>> H^2(G,R)\\
@. @V\eta\circ\_ VV @V\theta\circ\_ VV @V\zeta\circ\_ VV \\
@>>> H^1(G,Q')@>>> H^1(G,S')@>>> H^2(G,R')\end{CD}\end{align*}
of the two sequences arising from (i) above.
\item There is another short exact sequence  of $G$-groups 
  \[1 \to R \to Q_\gamma \to S_\gamma \to 1\]
with the image of $R$ contained in the centre of $Q_\gamma$. Moreover in the situation of (ii), there is a commutative diagram \[\begin{CD} 1@>>> R @>>> Q_\gamma @>>> S_\gamma @>>>  1\\
@. @V\zeta VV @V\eta VV @V\theta VV @. \\
1@>>> R' @>>> Q_{\gamma'}' @>>> S_{\gamma'}' @>>>  1.\end{CD}\]
where $\gamma'$ indicates the cocycle $\eta\circ\gamma$. 
\item Thus we get by (i) a new exact sequence of cohomology:
{\small \[1\to R^G \to Q_\gamma^G \to S_\gamma^G \to  H^1(G,R) \to H^1(G,Q_\gamma) \to H^1(G,S_\gamma)\to H^2(G,R_\gamma),\]} which is also functorial in the manner of (ii).
\item \label{bijofh1}The map \[\theta_\gamma:H^1(G,Q_\gamma)\to H^1(G,Q);\ [\delta]\mapsto[\delta\gamma],\]
where $\delta\gamma$ denotes the map $g\mapsto\gamma(g)\delta(g)$, is a well-defined bijection, taking the trivial class in $H^1(G,Q_\gamma)$ to the class of $\gamma$ in $H^1(G,Q)$.\setcounter{saveenum}{\value{enumi}}
\end{enumerate}
Now let $B$ be a subgroup of $G$. Then: 
\begin{enumerate}  \setcounter{enumi}{\value{saveenum}}
\item Restriction to $B$ of the exact sequence of cohomology from (i) gives rise to the following commutative diagram, where the rows are exact and the vertical arrows are restrictions: 
\begin{align*}\minCDarrowwidth19pt \begin{CD} 1@>>> R^G @>>> Q^G @>>> S^G @>\delta_G>>  H^1(G,R)\\
@. @VVV @VVV @VVV @VVV \\
1@>>> R^B @>>> Q^B @>>> S^B @>\delta_B>>  H^1(B,R) \end{CD}\hspace{80pt}\\
\hspace{80pt}\begin{CD} @>>> H^1(G,Q)@>>> H^1(G,S)@>>> H^2(G,R)\\
@. @VVV @VVV @VVV \\
@>>> H^1(B,Q)@>>> H^1(B,S)@>>> H^2(B,R).\end{CD}
\end{align*}
\item If $\beta\in Z^1(B,Q)$ denotes the restriction of $\gamma$ to $B$, then we have the following commutative diagram
\[\begin{CD} 1@>>> R @>>> Q_\gamma @>>> S_\gamma @>>> 1\\
@. @VVV @VVV @VVV @. \\
1@>>> R @>>> Q_\beta @>>> S_\beta @>>> 1.\end{CD},\] where the vertical arrows are restrictions. Moreover,
\item restriction from $G$ to $B$ gives rise to the following commutative diagram, where the rows are exact and the vertical arrows are restrictions:
\begin{align*}\minCDarrowwidth19pt \begin{CD} 1@>>> R^G @>>> Q_\gamma^G @>>> S_\gamma^G @>\delta_G>>  H^1(G,R)\\
@. @VVV @VVV @VVV @VVV \\
1@>>> R^B @>>> Q_\beta^B @>>> S_\beta^B @>\delta_B>>  H^1(B,R) \end{CD}\hspace{80pt}\\
\hspace{80pt}\begin{CD} @>>> H^1(G,Q_\gamma)@>>> H^1(G,S_\gamma)@>>> H^2(G,R)\\
@. @VVV @VVV @VVV \\
@>>> H^1(B,Q_\beta)@>>> H^1(B,S_\beta)@>>> H^2(B,R).\end{CD}
\end{align*}
\item We have the following commutative diagram:
\begin{center}
$\begin{CD}H^1(G,Q)@>>> H^1(G,S)\\
@A\theta_\gamma AA @A\theta_\beta AA\\
H^1(G,Q_\gamma) @>>> H^1(G,S_\beta),\end{CD}$\end{center} where the vertical maps are the bijections of (iv) and the horizontal maps are induced by the $G$-homomorphism $Q\to S$.
\item We also have the following commutative diagram
\[\begin{CD} H^1(G,Q_\gamma) @>\theta_\gamma>> H^1(G,Q)\\
@VVV @VVV\\
H^1(B,Q_\beta) @>\theta_\beta>> H^1(B,Q)\end{CD}\]
where the vertical maps arise from restriction.
\end{enumerate}
\end{prop}\label{restrict}
\begin{proof} \begin{enumerate}\item This is \cite[I.Prop. 38]{Ser94}.
\item is a straightforward check. For example, to see that  \[\xymatrix@!0{ 
S^G\ar@{>}[rrr]^{\delta_G} \ar@{>}[dd]^\theta & & & H^1(G,R)\ar@{>}[dd]^{\eta\circ}\\
\\
S'^G \ar@{>}[rrr]^{\delta'_G} & & & H^1(G,R').}\]
commutes: on elements we have
\[\xymatrix@!0{ 
qR\ar@{>}[rrrr] \ar@{>}[dd]& & & & &[q^{-\bullet}q]\ar@{>}[dd]\\
\\
\theta(qR)=\zeta(q)R' \ar@{>}[rrrrr] & & & & &[\zeta(q)^{-\bullet}\zeta(q)]=\zeta\left([q^{-\bullet}q]\right).}\]
 
\item is explained on \cite[p52]{Ser94}.
\item is clear from (i) and (iii).
\item is \cite[I.Prop. 35 bis]{Ser94}.
\item is easy to check from the definitions of the maps concerned. For instance, if $qR\in Q/R\cong S$ with $(qR)^g=qR$ for all $g$ then $(\text{res}^G_B\circ \delta_G)(qR)(b)=[q^{-\bullet}q](b)=q^{-b}q=\delta_B(qR)=\delta_B(\text{res}^G_B qR)$ for all $b\in B$.
\item is clear.
\item follows from (iv), (vi) and (vii).
\item is explained on \cite[p47]{Ser94}.
\item On elements 
\[\xymatrix@!0{ 
[\delta]\ar@{>}[rrrr] \ar@{>}[dd]& & & &[\delta\gamma]\ar@{>}[dd]\\
\\
[\delta|_B] \ar@{>}[rrrr] & & & &[\delta|_B\beta]=[\delta|_B\gamma|_B].}\]
\end{enumerate}
\end{proof}

Putting together the components of the last proposition:
\begin{prop}\label{cuboid}(i) With the hypotheses of Proposition \ref{list}, we have the following commutative partial cuboid
\[\xymatrix@!0{ 
& & H^1(G,R) \ar@{>}[rrrr]\ar@{>}'[dd][dddd]
& & & & H^1(G,Q) \ar@{>}'[dd][dddd] \ar@{>}[rrrr]
& & & & H^1(G,S) \ar@{>}[dddd]
\\ \\
H^1(G,R)\ar@{>}[rrrr]\ar@{>}[dddd]
& & & & H^1(G,Q_\gamma) \ar@{>}[uurr]\ar@{>}[dddd]\ar@{>}[rrrr] 
& & & & H^1(G,S_\gamma) \ar@{>}[uurr]\ar@{>}[dddd]
\\ \\
& & H^1(B,R) \ar@{>}'[rr][rrrr] 
& & & & H^1(B,Q) \ar@{>}'[rr][rrrr] 
& & & & H^1(B,S)
\\ \\
H^1(B,R) \ar@{>}[rrrr]
& & & & H^1(B,Q_\beta) \ar@{>}[rrrr]\ar@{>}[uurr] 
& & & & H^1(B,S_\beta) \ar@{>}[uurr]
},\]
where rightward arrows are part of four exact sequences running through the central vertical square.

(ii) With the hypotheses of Proposition \ref{list}(ii), we have the following commutative partial cuboid
\[\xymatrix@!0{ 
& & H^1(G,R) \ar@{>}[rrrr]\ar@{>}'[dd][dddd]
& & & & H^1(G,Q) \ar@{>}'[dd][dddd] \ar@{>}[rrrr]
& & & & H^1(G,S) \ar@{>}[dddd]
\\ \\
H^1(G,R)\ar@{>}[rrrr]\ar@{>}[dddd]
& & & & H^1(G,Q_\gamma) \ar@{>}[uurr]\ar@{>}[dddd]\ar@{>}[rrrr] 
& & & & H^1(G,S_\gamma) \ar@{>}[uurr]\ar@{>}[dddd]
\\ \\
& & H^1(G,R') \ar@{>}'[rr][rrrr] 
& & & & H^1(G,Q') \ar@{>}'[rr][rrrr] 
& & & & H^1(G,S')
\\ \\
H^1(G,R') \ar@{>}[rrrr]
& & & & H^1(G,Q_{\gamma'}') \ar@{>}[rrrr]\ar@{>}[uurr] 
& & & & H^1(G,S_{\gamma'}') \ar@{>}[uurr]
},\]
\end{prop}
\begin{proof}For the first partial cuboid, the front and back faces are subdiagrams of Proposition \ref{list}(vi),(viii); the top and bottom faces commute by (ix); the vertical squares commute by (x).

For the second partial cuboid, the front and back faces commute by Proposition \ref{list}(ii) and (iv). The vertical squares commute by  (ix).
\end{proof}

\section{Main results}
\subsection{Versions of the five lemma}
\begin{theorem}\label{fivelemmaforh1}Assume the hypotheses of Proposition \ref{list} and let $B$ be a subgroup of $G$. Then in the following subdiagram of Proposition \ref{list}(vi), 
\begin{center}
$\begin{CD} S^G @>\delta_G>>  H^1(G,R)@>\iota_G>> H^1(G,Q)@>\pi_G>> H^1(G,S)@>\Delta_G>> H^2(G,R)\\
@V h_1VV @V h_2VV @V h_3VV @V h_4VV @V h_5VV \\
 S^B @>\delta_B>>  H^1(B,R)@>\iota_B>> H^1(B,Q)@>\pi_B>> H^1(B,S)@>\Delta_B>> H^2(B,R),%%%%
\end{CD}$\end{center}
 the following hold:
\begin{enumerate}\item If $h_2$ and $h_4$ are surjective and $h_5$ is injective, then $h_3$ is surjective. 
\item If $h_2$ and $h_4$ are injective and the restriction maps $S_\gamma^G\to S_\beta^B$ are surjective for any $\gamma\in Z^1(G,S)$ with $\gamma|^G_B=\beta$, then $h_3$ is injective.
\item If the hypotheses of (i) and (ii) hold, then $h_3$ is an isomorphism. \end{enumerate}\end{theorem}
\begin{proof} Assume the hypotheses of (i) and take $\gamma\in H^1(B,Q)$. We need to produce a pre-image $\beta$, say, such that $h_3(\beta)=\gamma$. We get started by the diagram chase used to prove the five lemma for diagrams of abelian groups.

Let $\delta:=\pi_B(\gamma)$. As the bottom row is exact, $\Delta_B(\delta)=1$. As $h_4$ is surjective, we have a pre-image $\epsilon$ with $h_4(\epsilon)=\delta$.

Since the last square commutes, $h_5(\Delta_G(\epsilon))=1$ and since $h_5$ is injective, $\Delta_G(\epsilon)=1$.

Now, the top row is exact, so $\epsilon\in\ker \Delta_G$ and hence $\epsilon\in\im\pi_G$; say, $\epsilon=\pi_G(\eta)$. Let $h_3(\eta)=\theta$.

Since the penultimate square commutes, we have $\pi_B(\theta)=\delta$.

The picture is now as follows:

\begin{center}
$\begin{CD}  * @>\delta_G >>  *@>\iota_G>> ?,\eta@>\pi_G>> *,\epsilon@>\Delta_G>> 1\\
@V h_1 VV @V h_2 VV @V h_3 VV @V h_4 VV @V h_5 VV \\
 *@>\delta_B>>  *@>\iota_B>>{\gamma},{\theta}@>\pi_B>> \delta@>\Delta_B>> 1\end{CD}$\end{center}
where we would like to assert the existence of an element replacing `$?$' and  mapping to $\gamma$.

In the case of abelian groups, one would continue the proof of the five lemma by taking the difference $\gamma-\theta$; observing that this maps under $\pi_B$ to $1$ and continuing the diagram chase. Since we cannot do this in the case of pointed sets we use twisting to reduce to the case $\theta=1$. For this we use  Proposition \ref{cuboid}(i), replacing the back face with the middle of the diagram above. 

\[\xymatrix@!0{ 
& & {*}\ar@{>}[rrrr]^{\iota_G}\ar@{>}'[dd][dddd]_{h_2}
& & & & ?,\eta \ar@{>}'[dd][dddd]_{h_3} \ar@{>}[rrrr]^{\pi_G}
& & & & \epsilon \ar@{>}[dddd]_>>>>>>>{h_4}
\\ \\
?_2 \ar@{>}[rrrr]^>>>>>>{\iota_G'}\ar@{>}[dddd]_>>>>>>>{h_2}
& & & & ?_3,1 \ar@{>}[uurr]^{\theta_\eta}\ar@{>}[dddd]_>>>>>>>{h_3'}\ar@{>}[rrrr]^>>>>>>{\pi_G'}
& & & & 1 \ar@{>}[uurr]^{\theta_\epsilon}\ar@{>}[dddd]_>>>>>>>{h_4'}
\\ \\
& & {*}\ar@{>}'[rr]^{\iota_B}[rrrr] 
& & & & \gamma,\theta \ar@{>}'[rr]^{\pi_B}[rrrr] 
& & & & \delta
\\ \\
?_1,1\ar@{>}[rrrr]^{\iota_B'}
& & & & \gamma',1 \ar@{>}[rrrr]^{\pi_B'}\ar@{>}[uurr]^{\theta_\theta} 
& & & & 1 \ar@{>}[uurr]^{\theta_\delta}
},\]

Here we are using the fact that the bijection in Proposition \ref{list}(v) takes the trivial element in $H^1(G,Q_\eta)$ to the element $\eta$ in $H^1(G,Q)$ and the fact that the partial cuboid is commutative.

Now in the front bottom row, as $\gamma'\in\ker\pi_B'$, we have $\gamma'\in\im\ \iota_B'$. Thus we may replace $?_1$ with an element $\kappa$ mapping to $\gamma'$. Since $h_2$ is a surjection, we may replace $?_2$ with an element $\lambda$ mapping to $\kappa$. Then we replace $?_3$ with $\mu=\iota_G'(\lambda)$ and by the fact that the front left square commutes,  $h_3'(\mu)=\gamma'$. Finally if we replace $?$ with $\nu:=\theta_\eta(\mu)$ the commutativity of the central vertical square gives us our preimage of $\gamma$.

For (ii) the picture in the partial cuboid is as follows:
\[\xymatrix@!0{ 
& &
& & & & {*}\ar@{>}[rrrr]^{\iota_G}\ar@{>}'[dd][dddd]_{h_2}
& & & & \gamma,\delta \ar@{>}'[dd][dddd]_{h_3} \ar@{>}[rrrr]^{\pi_G}
& & & & {\zeta} \ar@{>}[dddd]_>>>>>>>{h_4}
\\ \\
1,\nu \ar@{>}[rrrr]^{\delta_G'}\ar@{>}[dddd]_{h_1'}
& & & & 1,\kappa \ar@{>}[rrrr]^>>>>>>{\iota_G'}\ar@{>}[dddd]_>>>>>>>{h_2}
& & & & 1,\delta' \ar@{>}[uurr]^{\theta_\gamma}\ar@{>}[dddd]_>>>>>>>{h_3'}\ar@{>}[rrrr]^>>>>>>{\pi_G'}
& & & & 1 \ar@{>}[uurr]^{\theta_\zeta}\ar@{>}[dddd]_>>>>>>>{h_4'}
\\ \\
& &
& & & & {*}\ar@{>}'[rr]^{\iota_B}[rrrr] 
& & & & \epsilon \ar@{>}'[rr]^{\pi_B}[rrrr] 
& & & & \eta
\\ \\
1,\mu \ar@{>}[rrrr]^{\delta_B'}
& & & & 1,\lambda\ar@{>}[rrrr]^{\iota_B'}
& & & & 1 \ar@{>}[rrrr]^{\pi_B'}\ar@{>}[uurr]^{\theta_\epsilon} 
& & & & 1 \ar@{>}[uurr]^{\theta_\eta}
}\]
where we have twisted by $\gamma$. Here a preimage $\kappa$  of $\delta'$ under $i'_G$ exists since $\delta'$ is in the kernel of $\pi'_G$. Simlarly $\mu$ is a preimage of $\lambda$ under $\delta'_B$; and $\nu$ is a preimage of $\mu$ under $h_1'$ using the fact that that map is surjective.
 
 This shows that $\delta'=\iota'_G(\delta'_G(\nu))$ and hence is equal to $1$ since the composition of these two maps is trivial. Thus $\gamma=\delta$ since $\theta_\gamma$ is a bijection.
 
(iii) is now obvious.
\end{proof}

An identical proof using Proposition \ref{cuboid}(ii) shows:

\begin{theorem}\label{fivelemmaforh1funct}Assume the hypotheses of Proposition \ref{list}(ii). Then in the following subdiagram of Proposition \ref{list}(ii), 
\begin{center}
$\begin{CD} S^G @>\delta_G>>  H^1(G,R)@>\iota_G>> H^1(G,Q)@>\pi_G>> H^1(G,S)@>\Delta_G>> H^2(G,R)\\
@V h_1VV @V h_2VV @V h_3VV @V h_4VV @V h_5VV \\
 {S'}^G @>\delta_G>>  H^1(G,R')@>\iota_B>> H^1(G,Q')@>\pi_B>> H^1(G,S')@>\Delta_B>> H^2(G,R').\end{CD}$\end{center}
 the following hold:
\begin{enumerate}\item If $h_2$ and $h_4$ are surjective and $h_5$ is injective, then $h_3$ is surjective. 
\item If $h_2$ and $h_4$ are injective and the restriction maps $S_\gamma^G\to {S'}_{\gamma'}^G$ are surjective for any $\gamma\in Z^1(G,S)$ with $\gamma'=\eta\circ\gamma$, then $h_3$ is injective.
\item If the hypotheses of (i) and (ii) hold, then $h_3$ is an isomorphism. \end{enumerate}\end{theorem}

\begin{prop}\label{fivelemmaforh0}Assume the hypotheses of \ref{list}(ii) and let $B$ be a subgroup of $G$. Then in each of the following diagrams:
\[\minCDarrowwidth19pt \begin{CD} 1@>>> R^G @>>> Q^G @>>> S^G @>\delta_G>>  H^1(G,R)\\
@V h_1 VV @V h_2 VV @Vh_3 VV @Vh_4 VV @Vh_5 VV \\
1@>>> R^B @>>> Q^B @>>> S^B @>\delta_B>>  H^1(B,R)\end{CD}\]
and
\[\minCDarrowwidth19pt \begin{CD} 1@>>> R^G @>>> Q^G @>>> S^G @>\delta_G>>  H^1(G,R)\\
@V h_1 VV @V h_2 VV @Vh_3 VV @Vh_4 VV @Vh_5 VV \\
1@>>> R'^G @>>> Q'^G @>>> S'^G @>\delta_G>>  H^1(G,R');\end{CD}\]
\begin{enumerate}\item If $h_2$ and $h_4$ are surjective and $h_5$ is injective, then $h_3$ is surjective.
\item If $h_2$ and $h_4$ are injective then $h_3$ is injective.
\item If $h_2$ and $h_4$ are isomorphisms  and $h_5$ is an injection then $h_3$ is an isomorphism.\end{enumerate}
\end{prop}
\begin{proof} The usual proof of the five lemma \cite[2.72]{Rot09} goes through in these cases. Where one would take the `difference' of two elements $g$ and $h$ in an abelian group  one uses the element $gh^{-1}$. The proof then works as normal.\end{proof}

\begin{corollary}\label{corGtoB}Suppose $Q$ is a $G$-group admitting a filtration by $G$-groups $Q=Q_1\geq Q_2 \geq Q_3 \geq \dots \geq Q_{n+1}=\{1\}$ with each $Q_i$ normalised by $G$, $Q_i\triangleleft Q$ and $Q_i/Q_{i+1}\leq Z(Q/Q_{i+1})$.

(a) If $B$ is a subgroup of $G$ and for all $1\leq j\leq n$ we have \begin{enumerate}\item $H^1(G,Q_j/Q_{j+1})\cong H^1(B,Q_j/Q_{j+1})$;
\item $H^0(G,Q_j/Q_{j+1})\twoheadrightarrow H^0(B,Q_j/Q_{j+1})$;
\item $H^2(G,Q_j/Q_{j+1})\hookrightarrow H^2(B,Q_j/Q_{j+1})$, \end{enumerate}
then $H^1(G,Q)\cong H^1(B,Q)$.

(b) Suppose $Q'$ is another $G$-group admitting another such filtration by $G$-groups $Q'=Q'_1\geq Q'_2 \geq Q'_3 \geq \dots \geq Q'_{n+1}=\{1\}$ and that there is a $G$-homomorphism $\rho:Q\to Q'$ such that $\rho(Q_i)\leq Q'_i$. If for all $1\leq j\leq n$ we have \begin{enumerate}\item $H^1(G,Q_j/Q_{j+1})\cong H^1(G,Q'_j/Q'_{j+1})$;
\item $H^0(G,Q_j/Q_{j+1})\twoheadrightarrow H^0(G,Q'_j/Q'_{j+1})$;
\item $H^2(G,Q_j/Q_{j+1})\hookrightarrow H^2(G,Q'_j/Q'_{j+1})$,\end{enumerate}then $H^1(G,Q)\cong H^1(G,Q')$.
\end{corollary}
\begin{proof} (a) This is a simple induction using the previous results. First, the fact that $Q_j$ is normalised by both $Q$ and $G$ mean that that $(Q_\gamma)_j:=(Q_j)_\gamma$ gives a filtration of $Q_\gamma$; we wish to show that the hypotheses (i) and (ii) give us: (*) For each $\gamma\in H^1(G,Q)$, we have $H^0(G,Q_\gamma/(Q_\gamma)_j)\twoheadrightarrow H^0(B,Q_\gamma/(Q_\gamma)_j)$.

Suppose (*) is known to hold for all $G$-groups with filtrations up to $n-1$, then we must just check that $Q^G\twoheadrightarrow Q^B$. So set $R=Q_n=(Q_\gamma)_n$, and $S=Q_\gamma/(Q_\gamma)_n$. The hypotheses of Proposition \ref{fivelemmaforh0}(i) now hold and we have $Q^G\twoheadrightarrow Q^B$ as required; proving (*). 

For the conclusion of (a) note that the case $n=1$ is immediate from the hypotheses. If (a) is known to hold for $G$-groups admitting a filtration of length up to $n-1$ then set $R=Q_n$, and $S=Q/Q_n$, then by (*) the hypotheses of  of Theorem \ref{fivelemmaforh1}(iii) hold. We conclude $H^1(G,Q)\cong H^1(B,Q)$.

(b) is similar, using \ref{fivelemmaforh1funct}(iii). 
\end{proof}

\begin{remark}\label{Bnotasub}The results in this section involving $B$ exist in a more general form which we do not need for the purposes of this paper. It is not necessary for $B$ to be a subgroup of $G$---merely that there is some homomorphism $\rho:B\to G$. Then, for instance, the `restriction maps' of $1$-cohomology become `composition with $\rho$'.\end{remark} 

\subsection{Algebraic $G$-groups}\label{algggrp}
Let $G$ be an algebraic group defined over an algebraically closed  field $k$ of characteristic $p\geq 0$. Let $Q$ be another algebraic group defined over $k$. Then we say $Q$ {\it is an algebraic $G$-group} if $Q$ is a $G$-variety (i.e. the map $G\times Q\to Q; (g,q)\mapsto q^g$ is a morphism of varieties), and $G$ acts as group automorphisms of $Q$. In this case, following \cite[8.4]{Hum75} one can form a semidirect product $GQ$ with multiplication defined by $g_1q_1.g_2q_2=g_1g_2q_1^{g_2}q_2$.

One checks that the results and proofs of the previous sections apply to the category of algebraic groups provided one is clear on the nature of this category: First, one insists that maps are morphisms of varieties, i.e. continuous in the Zariski topology. Thus by definition $Z^i(G,Q)$ is the set of {\it regular} maps $\gamma:\prod_1^iG\to Q$ satisfying the appropriate cocycle condition. This moreover ensures that $G$ acts regularly on $Q_\gamma$.\footnote{Of course one still has the original cohomology theory $H^i_\text{abs}$ defined via abstract cocycles, but we will not be considering this theory when $G$ is an algebraic group. }

With this change note that where we have made cocycle-coboundary definitions for $H^i(G,Q)$, these coincide with the usual rational (Hochschild) cohomology definitions of cohomology for algebraic groups in the case where $Q$ is a (rational) $G$-module, such as those used in \cite{CPSV77}. This is all worked out carefully in \cite[4.1]{McN10}. This means in particular that the results and proofs quoted in Proposition \ref{list} go over to our situation; this applies especially to results quoted from \cite{Ser94}.
%One has in \cite[I.4.14]{Jan03} for any $G$-module $M$, a Hochschild cohomology complex $C^n(G,M)=M\otimes \bigotimes^n k[G]$ with appropriate differentials, whose cohomology gives the cohomology groups $H^n(G,M)$ (\cite[I.4.16]{Jan03}). As is described in {\it loc. cit.}, $C^n(G,M)$ can be interpreted as the set $\Mor(G^n,M)$ of algebraic morphisms with the differentials as described in \cite[\S15.7]{Hal76}. Then the $0$-, $1$- and $2$-cocycle conditions are given in \cite[(15.7.3)]{Hal76}, concurring with our cocycle conditions in the abstract group case.

Second, by a subgroup we will mean always a closed subgroup; i.e. the image of an injection of algebraic groups. For instance, since a cocycle $\gamma$ in $Z^1(G,Q)$ is a morphism $\gamma:G\to Q$, the subgroup $\{g\gamma(g):g\in G\}$ of the semidirect product $GQ$ is now a closed subvariety of $GQ$. By a complement to $Q$ in $GQ$, we use the definition in \cite[4.3.1]{McN10}:
\begin{defn}\label{compdef} Let $G=H\ltimes Q$ be a semidirect product of algebraic groups as in \cite[I.2.6]{Jan03}.

A closed subgroup $H'$ of $G$ is a \emph{complement} to $Q$ if it satisfies the following equivalent conditions:
\begin{enumerate}\item Multiplication is an isomorphism $H'\ltimes Q\to G$.
\item $\pi_{H'}:H'\to H$ is an isomorphism of algebraic groups
\item Viewed as group-schemes, $H'Q=G$ and $H'\cap Q=\{1\}$.
\item For the groups of $k$-points, one has $H'(k)Q(k)=G(k)$, $H'(k)\cap Q(k)=\{1\}$ and $\Lie(H')\cap \Lie(Q)=0$.\footnote{The definition and equivalence of (i) and (ii) above are \cite[4.3.1]{McN10}. (i) is equivalent to (iii) using the remark at the end of \cite[I.2.6]{Jan03}. For (iv), $H'(k)Q(k)=G(k)\iff HQ=G$ and by \cite[I.7.9(2)]{Jan03}, $\Lie(H'\cap Q)=\Lie(H')\cap\Lie(Q)$, so (iii) is equivalent to (iv).}\end{enumerate}
\end{defn}

One checks that one retains the correspondence between complements and $1$-cocycles in this category:
\begin{lemma}\label{comp}The set of $1$-cocycles $Z^1(H,Q)$ is in bijection with the set of complements to $Q$ in $HQ$. Two cocycles are equivalent if the corresponding complements are conjugate by an element of $H(k)$.\end{lemma}
\begin{remark} This is a non-abelian generalisation from the case where $Q$ is an $H$-module. All the proofs are exactly the same. See \cite[4.5]{McN10}. \end{remark}

Lastly, where we consider extensions of algebraic $G$-groups we should emphasise the importance of the existence of the section $\sigma$ in Proposition \ref{list}.
\begin{defn} Let \[1\to R\to Q\stackrel{\pi}{\to} S\to 1\] be a short exact sequence of algebraic groups. We call $Q$ a {\it sectioned extension of $R$ by $S$} if there exists a regular map of varieties $\sigma:Q\to S$ with $\pi\circ\sigma$ the identity on $S$.\end{defn} Note that a sectioned extension is in particular a strongly exact sequence; that is to say that the map $d\pi:\Lie Q\to \Lie S$ is onto. If $Q$ is a sectioned extension of algebraic groups $R$ and $S$ with $R$ abelian, it follows that $Q$ corresponds to a (regular) $2$-cocycle class in $H^2(S,R)$: e.g. \cite[VII.4(a)]{Ser88}. (Obviously in the category of abstract groups, all extensions will be `sectioned', since they need simply to be defined set-wise.) We record the obvious fact:
\begin{lemma}\label{components}For any algebraic group $G$ and any $G$-group $Q$, the extension of $G$-groups \[1\to Q^\circ\to Q\to Q/Q^\circ\to 1\] is sectioned.\end{lemma}

\begin{remark}Much of the cohomology theory for algebraic groups predicated on regular maps fails if one considers non-sectioned extensions of algebraic groups. We thank Steve Donkin for the following example. Let $R=\F_p$,  $Q=\G_a$ and $S=Q/\F_p$, where $\G_a$ is the 1-dimensional additive group. (Then $S$ is isomorphic to $\G_a$ too.) Imagine there were a morphism of varieties $\sigma:S\to Q$ with $\sigma$ a section of the quotient map $Q\to S$. Then $q+\F_p=\pi\sigma(q+\F_p)=\sigma(q+\F_p)+\F_p$. So $\sigma(q+\F_p)-q\in \F_p$. Since $S$ is connected, $F_p$ is discrete and $\sigma$ is continuous, the image of this map must be constant, hence $0$. So $\sigma(q+\F_p)$ is constantly $q$. If $0\neq r\in F_p$ then $\sigma(q+r+\F_p)=q=q+r$. So $r=0$; a contradiction.

Indeed the extension \[1\to R\to Q\to S\to 1\] cannot correspond to a $2$-cocycle $\phi:S\times S\to R$ since such regular cocycles from a continuous to a discrete group must again be trivial.

Worse, letting $G\cong S$, no such exact sequence of Proposition \ref{list}(i) can exist using rational cohomology: if it did, one would have
\[H^1(G,R)\to H^1(G,Q)\to H^1(G,S)\to H^2(G,R),\]
with $H^i(G,R)=0$ for $i=1,2$, as continuous maps $\gamma:G\to R$ or $\gamma:G\times G\to R$ must be constant (and so cocycles are zero maps). Thus we would have $H^1(G,Q)\cong H^1(G,S)$. The latter contains the identity map, but the above argument also shows there can be no preimage; a contradiction.\end{remark}

Happily, sectioned extensions exist in some general settings as in the lemma below. We need some further definitions:

\begin{defn} Let $G$ be a group and $Q$ a $G$-group. We say that a (finite) filtration \[Q=Q_1\geq Q_2\geq\dots \geq Q_{n+1}=1\] of $Q$ with $Q_i\triangleleft Q$ is {\it central} if $Q_i/Q_{i+1}\leq Z(Q/Q_{i+1})$;
we say that it is {\it a filtration by $G$-groups} if each $Q_i$ is normalised by $G$; we say it is {\it a filtration by $G$-modules} if it is a filtration by $G$-groups and for each $i$, $Q_i/Q_{i+1}$ has the structure of a (rational, linear) $G$-module; finally, we say that it is {\it sectioned} if for each $i$ there is a map $Q/Q_i\to Q$.\end{defn}

The following is a result of Rosenlicht.
\begin{lemma}[{c.f. \cite[14.2.6]{Spr98}}]\label{sectioned}Let $Q$ be a connected unipotent algebraic subgroup of an algebraic group $G$. Then \begin{enumerate}\item there exists a (regular) section to the morphism $G\to G/Q$ of varieties;
\item $G$ is isomorphic (as a variety) to $G/Q\times Q$.\end{enumerate}
In particular, any filtration of $Q$ by %%%%
connected unipotent subgroups is sectioned.\end{lemma}

\subsection{Unipotent algebraic $G$-groups}
%%%%
Let $G$ be a connected algebraic group. 
%%%%
Here we prove that any connected unipotent algebraic $G$-group $Q$ admits a (sectioned) central filtration by $G$-modules. This allows us to use results like Corollary \ref{corGtoB}. First we need four technical lemmas.

\begin{lemma}\label{m0ormv}Let $G$ be an algebraic group defined over $k$ and let $M$ be a closed $G$-stable subgroup of a finite dimensional simple $G$-module $V$. Then $M=V$ or $M$ is finite.\end{lemma}\begin{proof} Since $V$ is a $G$-module, there is a $G$-equivarient isomorphism between $V$ and $T_0(V)$, its tangent space at the identity. Under this map, $M$ is sent to a $G$-subspace $T_0(M)$ of $T_0(V)\cong V$, which is either $0$ or $V$ since $V$ is simple. If $T_0(M)=V$ then $M=V$ by dimensions; or if $T_0(M)=0$, then $M$ is $0$-dimensional, so finite.\end{proof}

Recall that a vector group $V$ is an algebraic group isomorphic to a direct product of $n$ copies of $\G_a$, the additive group of the field.

\begin{lemma}\label{vectgrp}Let $U$ be a connected unipotent algebraic group over a field $k$ of characteristic $p>0$, $V$ a vector group over $k$ and let $\pi:U\to V$ be an isogeny. Then $U$ is a vector group, hence isomorphic to $V$.
\end{lemma}
\begin{proof}If $u\in U$, $\pi(u^p)=p\pi(u)=0$, so the map $u\mapsto u^p$ is a continuous map from $U$ to the kernel of $\pi$. Such a map must be constant as $U$ is connected, thus identically $0$. Hence $u^p=0$. By \cite[Proposition 11]{Ser88}, $U$ is thus a vector group; dimensions imply that $U\cong V$.
\end{proof}

\begin{lemma}\label{absorbfinite}Let $G$ be a connected algebraic group and let \[1\to F\to W\stackrel{\pi}{\to} V\to 1\] be a short exact sequence of algebraic $G$-groups with $W$ a vector group, $V$ a simple $G$-module and $F$ a finite group. Then either $W$ has the structure of a trivial $G$-module or $F=0$. 

In any case, $W$ is isomorphic to $V$ as a $G$-module.\end{lemma}
\begin{proof}We have that $R_u(G)$ acts trivially on simple $G$-modules, so $V$ inherits the structure of a simple $H$-module for the reductive group $H=G/R_u(G)$. Take a maximal torus $T$ of $H$.

If $T$ acts trivally on $V$ then as $V$ is simple for $H$, we have $V\cong k$. Hence, by Lemma \ref{vectgrp}, $W\cong k$ also, with $G$ acting trivially and so $W$ has the structure of a trivial $G$-module.

Otherwise, let $v^+\in V$ be a highest weight vector for $H$ with non-trivial $T$-weight. Denote by $\langle Gv^+\rangle$ the closed subgroup of the (additive) group $V$ generated by elements $gv^+$. We claim $V=\langle Gv^+\rangle$. To see this, observe that $Tv^+=k^\times v^+$, and since $0\in \langle Gv^+\rangle$, we have $kv^+\leq \langle Gv^+\rangle$. Hence for each  $k_i\in k$ and $g_i\in G$, we have $k_ig_iv^+\in\langle Gv^+\rangle$. Then if $m=k_1g_1+\dots+k_lg_l\in kG$, we have $mv^+\in \langle Gv^+\rangle$. Thus $\langle Gv^+\rangle=kGv^+=V$. This proves the claim.

Now take $K=\left(\pi^{-1}(kv^+)\right)^\circ$ and let $0\neq w^+\in K$. We have $Tw^+$ is a connected subset of $\pi^{-1}(kv^+)$, so $\langle Tw^+\rangle=K$ and $K$ is a closed, connected $1$-dimensional subgroup of $W$, hence $K\cong k$, $K=Tw^+\cup \{0\}$ and $\pi(K)=kv^+$. The image of the map $\phi:G\times K\to W$ by $(g,x)\to gx$ is the image of an irreducible variety containing the identity of $W$. Hence by \cite[Prop.~7.5]{Hum75}, $\langle GK\rangle=\langle GTw^+\rangle=\langle Gw^+\rangle$ is a connected subgroup of $G$ whose elements all have the form $g_1w^++\dots+g_nw^+$. Moreover, $\pi\langle Gw^+\rangle=\langle Gv^+\rangle$, which is in turn equal to $V$ by the above claim, so we have that $\langle GK\rangle$ is a connected subgroup of $W$ isogenous to $V$. Thus $\langle Gw^+\rangle=W$ by dimensions and for any element $w\in W$ we can write $w=g_1w^++\dots+g_nw^+$. So let $f\in F\leq W$ and write $f=g_1w^++\dots+g_nw^+$. The map $\rho:K\to W$ by $x\mapsto g_1x+\dots+g_nx$ is the image of a connected $1$-dimensional variety, hence is $1$-dimensional or constant. Now, for any $\lambda\in k$, we have $g_1\lambda v^++\dots+g_n\lambda v^+=\lambda\pi(f)=0$, so that $\pi\rho(K)=0$. But since $\pi$ has only finite kernel, the map $\rho$ must have constant image. Since $0$ is in the image of $\rho$, it now follows that $f=0$.
\end{proof}

\begin{lemma}\label{nofinitebits}Let $Q$ be a connected unipotent algebraic $G$-group with a (central) filtration $\{Q_i\}$ such that successive quotients have the structure of simple $G$-modules or finite groups. Then $Q$ has a (central) filtration by $G$-modules.\end{lemma}
\begin{proof}%The idea here is to absorb all the finite subquotients into connected subquotients by pulling them upwards inductively.

Let $m:=|\{i:Q_i/Q_{i+1}\text{ is finite}\}|$. Take $j$ minimum such that $A:=Q_j/Q_{j+1}$ is finite (note that $j>1$ or else $Q$ has a disconnected quotient). Then $C:=Q_{j-1}/Q_j$ is a $G$-module. Letting $B=Q_{j-1}/Q_{j+1}$, we have the following complex induced from the short exact sequence $1\to A\to B\to C\to 1$:
\[1\to A\cap B^\circ\stackrel{\iota}{\to} B^\circ\stackrel{\pi}{\to} C\to 1.\tag{*}\]
Clearly $\iota$ is an injection; the image of $B^\circ$ under $\pi$ is a connected subgroup of $C$ of the same dimension as $C$ hence is equal to $C$, showing that $\pi$ is surjective. If $b\in\ker \pi$ then we obviously have $b\in A\cap B^\circ$; thus the complex is a short exact sequence. 

Now, $B^\circ$ is a connected unipotent group isogenous to the vector group $C$. By Lemma \ref{vectgrp}, $B^\circ\cong C$ as a vector group. Thus we may apply Lemma \ref{absorbfinite} to the short exact sequence (*) to see that $B^\circ$ is a simple $G$-module. Thus $B$ has a filtration $1\to B^\circ \to B\to B/B^\circ\to 1$ with $B^\circ$ a $G$-module and $B/B^\circ$ finite. Let $\{Q'_i\}$ be the filtration of $Q$ with $Q'_i=Q_i$ for all $i\neq j, j-1$ and $Q_{j-1}=B/B^\circ$, $Q_j=B^\circ$. Then $\{Q'_i\}$ is a filtration of $Q$ by $G$-modules or finite groups with $j-1$ as the minimum value of $i$ such that $Q_{i}/Q_{i+1}$ is finite. By induction we get a filtration $\{R_i\}$ of $Q$ with $Q/R_2$ finite. This is the image of a connected group and is hence trivial. Thus $\{R_i\}_{i\geq 2}$ is a (central) filtration of $Q$ by $G$-modules and finite groups with $m-1$ successive quotients being finite.

Repeating this process, we get a central filtration $Q$ by $G$-modules with no finite quotients.\end{proof}

We are now in a position to prove our second main result.

\begin{theorem}\label{gmodfilt}Let $G$ be a connected algebraic group over an algebraically closed field $k$ and $Q$ a connected unipotent $G$-group. Then $Q$ has a (sectioned) central filtration by $G$-modules.\end{theorem}
\begin{proof}Form the semidirect product $H=GQ$ and embed it as a closed subgroup of $GL_n(k)$, in such a way that $Q$ is represented by strictly upper triangular matrices. Write $M_n(k)$ for the algebra of $n\times n$ matrices	 over $k$ and regard $M_n(k)$ as a $G$-module via matrix conjugation. Let $A$ be the subalgebra of $M_n(k)$ generated by $Q\leq GL_n(k)\subset M_n(k)$. Then every element in $Q$, and therefore $A$, is $\lambda.I + x_n$ for $x_n$ nilpotent; thus $A$ is a local algebra with radical $J$, say; in fact, $J$ is the ideal given by all the nilpotent elements of $A$. Now $G$ acts on $A$ by algebra automorphisms; thus $G$ also acts as a group of automorphisms of each power $J^i$ of $J$ since each is a characteristic ideal. Equally, $G$ acts on each of the multiplicative groups $1+J^i$, again by matrix conjugation. Note that for $j_1\in J^i$, and $j_2\in J$ we have $(1+j_1)(1+j_2)=(1+j_2)(1+j_1)$ modulo $J^{i+1}$ and so $(1+J^{i})/(1+J^{i+1})\leq Z((1+J)/(1+J^{i+1}))$. Moreover, there is a $G$-equivarient automorphism from the additive $G$-group $J^i/J^{i+1}$ to the multiplicative $G$-group $(1+J^i)/(1+J^{i+1})$. So $R=1+J$ has a central filtration by $G$-modules. We may refine this to a filtration $R=R_0>R_1>R_2>\dots>1$ with consecutive quotients being simple $G$-modules.

Now, since $Q$ was unipotent, $Q-1\subset A$ consists of nilpotent matrices and therefore contained in $J$. So $Q\subseteq R=1+J$. Thus, letting $Q_i=Q\cap R_i$, we have a filtration of $Q$ with successive quotients being either simple $G$-modules, or finite $G$-groups by \ref{m0ormv}. By Lemma \ref{nofinitebits}, there is now a central filtration of $Q$ by $G$-modules.

Recall that such a filtration is automatically sectioned by Lemma \ref{sectioned}.\end{proof}
\subsection{Non-abelian cohomology of a group extension}

The following is a generalisation (by a small amount) of the five-term exact sequence from the Lyndon-Hochschild-Serre spectral sequence to our non-abelian situation. Here our groups can be algebraic or abstract, bearing in mind the impositions made in \S\ref{algggrp}

\begin{lemma}\label{lhs} Let $1\to U\to G\stackrel{\pi}{\to} T{\to} 1$ be an extension of groups and let $Q$ be a $G$-group. Then there is an exact sequence
\[1\to H^1(T,Q^U)\stackrel{\rm inf}{\to} H^1(G,Q)\stackrel{\rm res}{\to} H^1(U,Q)^T.\]
Moreover, \begin{enumerate}\item if $H^1(T,Q_\gamma^U)=\{1\}$ for all $\gamma\in Z^1(G,Q)$, the map $H^1(G,Q)\stackrel{\rm res}{\to} H^1(U,Q)$ is injective;\\\item if $H^1(U,Q_\gamma)^T=\{1\}$ for all $\gamma\in Z^1(G,Q)$ then $H^1(T,Q^U)\stackrel{\rm inf}{\to} H^1(G,Q)$ is an isomorphism of pointed sets.\end{enumerate}\end{lemma}
\begin{proof} The existence of the exact sequence can be found in \cite[I.5.8(a)]{Ser94} and \cite[6.2.3]{Ric82}.

Assuming the hypotheses for (i), take $\delta$ and $\gamma$ in $Z^1(G,Q)$ with $\delta(u)=q^{-u}\gamma(u)q$ for each $u\in U$. We wish to show $\delta\sim\gamma$. First, replacing $\gamma$ by the cocycle $g\mapsto q^{-g}\gamma(g)q$ we have $\delta(u)=\gamma(u)$. Now consider the exact sequence
\[1\to H^1(T,Q_\gamma^U)=1\stackrel{\rm inf}{\to} H^1(G,Q_\gamma)\stackrel{\rm res}{\to} H^1(U,Q_\gamma)^T.\]
The image of $\delta$ under the map $\theta_\gamma^{-1}$ given in Proposition \ref{list}(v) is now in the kernel of res. Thus $\theta_\gamma^{-1}(\delta)=1$ by exactness of this sequence. So $\delta=\theta_\gamma(1)=\gamma$.

Statement (ii) is similar.\end{proof} 

\begin{question}If, moreover, $Q^U$ is abelian, can one extend this sequence to the right with the term \[\to H^2(T,Q^U)?\]\end{question}

The lemma has the following consequences for algebraic groups.

\begin{corollary}\label{cartanalg} Let $B=U\rtimes T$ be a connected solvable algebraic group with unipotent radical $U$ and maximal torus $T$. Then if $Q$ is a unipotent $B$-group, the restriction map $H^1(B,Q)\to H^1(U,Q)$ is an injection.

Let $G$ be a connected, reductive group with $G_s=G/Z(G)^\circ$ a semisimple quotient. If $Q$ is a unipotent $G$-group we have $H^1(G,Q)\cong H^1(G_s,Q^{Z(G)^\circ})$.  \end{corollary}
\begin{proof} We have $H^1(T,Q^U)%%%%
=\{1\}$, for {\it any} unipotent group $Q$, by \cite[Lemma 6.2.6]{Ric82}, or an easy application of \ref{list}(i), \ref{gmodfilt} and the fact that we have $H^1(T,V)=0$ for all $T$-modules and the finite $T$-group $Q/Q^\circ$ (as $T$ is linearly reductive and connected, respectively). The hypotheses of Lemma \ref{lhs}(i) hold and the result follows.

For the second part, since $Z(G)^\circ$ is a torus, we also have $H^1(Z(G)^\circ,Q)=0$ for any unipotent algebraic group $Q$. Thus the statement follows from Lemma \ref{lhs}(ii). \end{proof}

\subsection{Application: $1$-cohomology of reductive $G$ with coefficients in a unipotent group}

In this section, $G$ will be a connected, reductive group defined over an algebraically closed field $k$ and $B$ a Borel subgroup of $G$. Recall that for any parabolic subgroup $P$ of $G$, there is a short exact sequence
\[1\to R_u(P)\to P\stackrel{\pi}\to L\to 1\]
where $R_u(P)$ is the unipotent radical of $P$, and $L$ is a complement to $R_u(P)$ in $P$, i.e. a Levi subgroup. We have a splitting homomorphism $\iota:L\to P$ such that $\pi\circ\iota$ is the identity map on $L$.

In \cite{CPSV77}, Cline, Parshall, Scott and van der Kallen proved a result which has the following generalisation

\begin{theorem}[{\cite[II.4.7]{Jan03}}]\label{cps}Let $G$ be a connected, reductive algebraic group over $k$ and let $P$ be a parabolic subgroup of $G$. If $V$ is a rational $G$-module, then $H^n(G,V)\cong H^n(P,V)$ for all $n\geq 0$.\end{theorem}

Using the results of the previous sections we will generalise this result in the case $n=1$, by replacing $V$ by an arbitrary unipotent $G$-group $Q$. 

\begin{theorem}\label{alggrp} Let $G$ be a connected, reductive algebraic group over an algebraically closed field $k$ and let $Q$ be a unipotent $G$-group. Let $P$ be a parabolic subgroup of $G$. Then the restriction map $H^1(G,Q)\to H^1(P,Q)$ is an isomorphisms of pointed sets. \end{theorem}
\begin{proof}We have $H^1(G,Q)\cong H^1(G,Q^\circ)$ and $H^1(P,Q)\cong H^1(P,Q^\circ)$ since $G$ and $P$ are connected. Also, by \ref{gmodfilt} the connected component $Q^\circ$ has a sectioned central filtration by (rational) $G$-modules. Now if $V$ is a $G$-module, $H^i(G,V)\cong H^i(P,V)$ by \ref{cps} for $0\leq i\leq 2$. We have the hypotheses for \ref{corGtoB} and so we conclude that $H^1(G,Q)\cong H^1(P,Q)$.
\end{proof}

\begin{remark}A more general result when $n=0$ is clear: for any affine $G$-variety $X$ we may choose a $G$-equivarient embedding of $X$ into a $G$-module $W$ with fixed points $X^G=X\cap W^G=X\cap W^B=X^B$.\end{remark}

To prove our second generalisation of results in \cite{CPSV77}, let us recall (and extend) the notion of a Frobenius twist---if $G$ is defined over a field $k$ of characteristic $p$ one has Frobenius morphisms $F_r=F_1^r:G\to G^{(r)}$ induced by the field automorphism $\sigma_q:k\to k$ by $x\mapsto x^{p^r}$ (see \cite[I.9.2]{Jan03}). Where $G$ is defined over the prime field $k_0\leq k$ we have $F_r(G)\cong G$. In this case, if $Q$ is any algebraic $G$-group (for instance a $G$-module) then one can define a new $G$-group $Q^{[r]}$ which is identical to $Q$ as a group, but where we twist the $G$-action by precomposing with $F_r$. If both $Q$ and the action map $G\times Q\to Q$ are also defined over $k_0$, one has a $G$-equivariant map of algebraic groups given by $Q\to Q^{[r]}$; call this the twist map. In the case $Q$ is a $G$-module, it is also $k$-semilinear. Denote the fixed points $G(k)^{F_r}$ of the group of $k$-points $G(k)$ by $G(q)$. If $G$ is connected reductive, $G(q)$ is called a finite reductive group.

The following is an easy improvement of the main result \cite[6.6]{CPSV77} to the case where $G$ is connected and reductive. We will only need it in the case that $m=1$.

\begin{lemma}\label{CPSV}Let $G$ be a connected reductive algebraic group defined and split over the prime field $k_0$ of $k$ with root system $\Phi$. Let $V$ be a rational $G$-module. Then there exist non-negative integers $e_0=e_0(\Phi,m,V)$ and $f_0=f_0(\Phi,m,V)$ so that for all $e\geq e_0$ and $f\geq f_0$, the restriction map $H^n(G,V^{[e]})\to H^n(G(p^{e+f}),V^{[e]})$ is an isomorphism for $n\leq m$ and an injection for $n=m+1$.\end{lemma}
\begin{proof}Let $Z=Z(G)^\circ$ be the connected centre of $G$. Set $G_s:=G/Z$, a semisimple quotient of $G$ and set $W=V^Z$, the $Z$-fixed points of $V$.

Invoking \cite[6.6]{CPSV77}, we find integers $e'_0=e'_0(\Phi,m,W)$ and $f'_0=f'_0(\Phi,m,W)$ so that for all $e\geq e'_0$ and $f\geq f'_0$, the restriction map $H^n(G_s,W^{[e]})\to H^n(G_s(p^{e+f}),W^{[e]})$ is an isomorphism for $n\leq m$ and an injection for $n=m+1$.

We claim there exists an $e_0\geq e'_0$ and $f_0\geq f'_0$ so that for all $e\geq e_0$ and $f\geq f_0$ we have (a) $H^n(G,V^{[e]})\cong H^n(G_s,W^{[e]})$; and (b) $H^n(G(p^{e+f}),V^{[e]})\cong H^n(G_s(p^{e+f}),W^{[e]})$ for all $n\geq 0$. Clearly this will finish the proof.

To prove the claim, take $e_0=e_0'$ and $f_0\geq f_0'$ so that each $Z$-weight of $V$ is $p^{f_0}$-restricted. Then for all $f\geq f_0$, and $q=p^{e+f}$ we have $(V^{[e]})^{Z(q)}=(V^{[e]})^Z$. We also have that $H^n(Z(q),V^{[e]})=0=H^n(Z,V^{[e]})$ for all $n\geq 1$, the former equality since $|Z(q)|$ is coprime to $p$, the latter as $Z$ is connected.

Now the Lyndon--Hochschild--Serre spectral sequence applied to $Z\triangleleft G$ degenerates to a line, yielding isomorphisms $H^n(G,V^{[e]})\cong H^n(G/Z,(V^{[e]})^Z)=H^n(G_s,W^{[e]})$ for all $n\geq 0$. This proves (a). For (b),  the Lyndon--Hochschild--Serre spectral sequence applied to $Z(q)\triangleleft G(q)$ similarly degenerates yielding isomorphisms $H^n(G(q),V^{[e]})\cong H^n(G(q)/Z(q),(V^{[e]})^{Z(q)})=H^n(G(q)/Z(q),W^{[e]})$ for all $n\geq 0$. Furthermore, as $Z$ is connected, we have by \cite[3.13]{DM91} that $G(q)/Z(q)\cong (G/Z)(q)=G_s(q)$. Thus $H^m(G(q),V^{[e]})\cong H^m(G_s(q),W^{[e]})$. This proves (b). The claim is proved and this finishes the proof of the lemma.\end{proof}

\begin{remark}\label{gmodfiltoverk0}In the proof of the next theorem, we will need to know that if $G$, $Q$ and the action map $G\times Q\to Q$ are defined over $k_0$ then the filtration of $Q$ asserted by Theorem \ref{gmodfilt} is also defined over $k_0$. We outline the embellishments to the proof. Note from \cite[2.3.7(ii)]{Spr98} that the embedding $GQ\hookrightarrow GL_n(k)=GL(V)$ can be taken to be defined over $k_0$; according to an appropriate $k_0$-basis of $V$ we can also insist $Q\hookrightarrow U_n(k)$. The remainder goes through fine: for instance $A=(k_0Q)\otimes_{k_0} k$, similarly for $J$ and all of its powers; hence also for the filtration $\{Q_i\}$. As $k_0$ is perfect, by \cite[12.2.1]{Spr98} quotients are defined over $k_0$; in particular the quotients $Q_i/Q_{i+1}$. Now the invocation of Lemma \ref{nofinitebits} under these hypotheses does yield another filtration over $k_0$: $B^\circ$ is defined over $k_0$ if $B$ is; also intersections of $k_0$-groups are defined over $k_0$ by \cite[12.1.15]{Spr98}. In conclusion, we see that $Q$ has a central filtration $\{Q_i\}$ with each $Q_i$ Frobenius stable, and $Q_i/Q_{i+1}$ is a $G$-module.\end{remark}

As in the case of a $G$-module, one always has an injection $H^1(G,Q)\hookrightarrow H^1(G,Q^{[1]})$, for instance, by looking at the cocycle-coboundary definitions. As in Lemma \ref{CPSV}, repeating this process quickly yields isomorphisms: 

\begin{theorem}\label{rationalstab} Let $G$ be a connected, reductive algebraic group defined over an algebraically closed field $k$ of characteristic $p>0$ and let $Q$ be a unipotent $G$-group with $G$, $Q$ and the action map $G\times Q\to Q$ defined over the prime field $k_0$ of $k$. Then there exists an integer $r_0$ such that for each $r\geq r_0$, the restriction map $H^1(G,Q^{[1]})\to H^1(G(p^r),Q^{[1]})$ is an isomorphism. Moreover this holds with $Q^{[1]}$ replaced by $Q$ unless $p=2$, $G$ has a component of type $C_l$ with $l\geq 1$ and $Q$ contains the $2l$-dimensional `natural representation' $V(\omega_1)$ as a $G$-composition factor.

In particular $H^1(G,Q^{[1]})\cong H^1(G,Q^{[e]})$ for $e\geq 1$. Moreover $H^1(G,Q)\cong H^1(G,Q^{[e]})$ for $e\geq 0$ unless $p=2$, $G/Z(G)^\circ$ contains  a factor of type $C_l$ and $Q$ contains the $2l$-dimensional $G$-composition factor $V(\omega_1)$.
\end{theorem}
\begin{proof}Assume first $Q$ is connected; so we may use Remark \ref{gmodfiltoverk0} to take a Frobenius stable central filtration $\{Q_j\}$ of $Q$ by $G$-modules. First assume we are not in the special case: that is where $p=2$, $G$ contains a factor of type $C_l$ and for some $j$, $Q_j/Q_{j+1}$ has a $G$-composition factor isomorphic to $V(\omega_1)$. We will show $H^1(G,Q)\cong H^1(G,Q^{[1]})$.

As in Lemma \ref{CPSV} one can generalise \cite[7.1]{CPSV77} to show that since we are not in the exceptional case, for each $1\leq j\leq n$, we have $H^1(G,Q_j/Q_{j+1})\cong H^1(G,(Q_j/Q_{j+1})^{[1]})$. One also has $H^2(G,Q_j/Q_{j+1})\hookrightarrow H^2(G,(Q_j/Q_{j+1})^{[1]})$, since the map $H^m(G,V)\to H^m(G,V^{[1]})$ is always an injection. Finally, it is clear that $H^0(G,Q_j/Q_{j+1})\cong H^0(G,(Q_j/Q_{j+1})^{[1]})$. We have the hypotheses for \ref{corGtoB}(b) and so we conclude that $H^1(G,Q)\cong H^1(G,Q^{[1]})$. 

Now, Lemma \ref{CPSV} asserts the existence, for each $G$-module $Q_j/Q_{j+1}$ an integer $r_j$ such that for $r\geq r_j$ restriction induces  isomorphisms $H^i(G,Q_j/Q_{j+1})\to H^i(G(p^r),Q_j/Q_{j+1})$ for $i=0,1$, and an injection $H^2(G,Q_j/Q_{j+1})\to H^2(G(p^r),Q_j/Q_{j+1})$. Take $r_0=\max_jr_j$. Then for any $r\geq r_0$,  take $B=G(p^r)$; we see the hypotheses of \ref{corGtoB}(a) hold and we conclude $H^1(G,Q)\cong H^1(G(p^r),Q)$.

Now we deal with the case where $Q$ is not connected; note again that $Q$ and $Q/Q^\circ$ are defined over $k_0$. We have $H^1(G,Q)\cong H^1(G,Q^\circ)$ as $G$ is connected. 

We claim there is also an isomorphism  $H^1(G(q),Q^\circ)\cong H^1(G(q),Q)$ for all sufficiently large $q$. Clearly this will now finish the proof in the non-exceptional case of the lemma. To prove the claim, from the short exact sequence \[1\to Q^\circ\to Q\to Q/Q^\circ\to 1\] we deduce that in the exact sequence of Proposition \ref{list}(i) the map $(Q/Q^\circ)^G=Q/Q^\circ\to H^1(G,Q^\circ)$ is trivial. We have shown that there exists an integer $r_0$ with $H^1(G,Q^\circ)\to H^1(G(q),Q^\circ)$ if $q=p^r$ and $r\geq r_0$. Thus using the commutative diagram of of \ref{list}(vi) with $G(q)$ taking the role of $B$, we see the map $(Q/Q^\circ)^G\to H^1(G,Q^\circ)$ is also trivial. We now wish to show $H^1(G(q),Q/Q^\circ)=\{1\}$. Since $G$ acts trivially on $Q/Q^\circ$, so does $G(q)$. As $Q/Q^\circ$ is a $p$-group, we can take a filtration of it by subgroups isomorphic to $\mathbb F_p$ under addition. Now $q$ is already big enough so that $H^1(G(q),k)=0$ %%%%
using Lemma \ref{CPSV}. Since $H^1(G(q),\mathbb F_p)\to H^1(G(q),k)=0$ is clearly an injection, we have $H^1(G(q),Q/Q^\circ)=\{1\}$ by applying the long exact sequence of Proposition \ref{list}(i) inductively through this filtration of $Q/Q^\circ$. One last application of Proposition \ref{list}(i) applied to $1\to Q^\circ\to Q\to Q/Q^\circ\to 1$ yields the exact sequence $1\to H^1(G(q),Q^\circ)\to H^1(G(q),Q)\to 1$. An easy twisting argument now shows that $H^1(G(q),Q^\circ)\to H^1(G(q),Q)$ is an isomorphism of pointed sets, thereby proving the claim.

In the exceptional case %%%%
of the theorem when $p=2$, exactly the same proof shows that $H^1(G,Q^{[1]})\cong H^1(G,Q^{[2]})$, using the exceptional case from \cite[7.1]{CPSV77}.%%%% and the fact that for a simple algebraic group $G$, $H^1(G,V)\cong H^1(G,V^{[1]})$ for a simple module $V$ unless $p=2$, $G=Sp_{2n}$ and $V=V(\omega_1)$.\footnote{This well-known fact follows from \cite[II.12.2, Remark]{Jan03} and a five-term exact sequence applied to $G_1\triangleleft G$.}
\end{proof}
\begin{defn} Following \cite[6.7(b)]{CPSV77}, we denote by $H^1_\text{gen}(G,Q)$ the stable value of $H^1(G(p^r),Q)$.\end{defn}
\begin{remark}\label{explicit}If one knows the weights of $G$ on a filtration of $Q$, one can establish an explicit value of $r_0$ in the theorem above using the proof of \ref{CPSV} and the explicit constants given for $e_0$ and $f_0$ in \cite{CPSV77}.\end{remark}

\subsection{Subgroup Structure}

Finally, we give a corollary on subgroup structure. For that we need the following lemma, which pins down when abstract complements in semidirect products of connected, reductive groups with unipotent groups are complements also as algebraic groups.

\begin{lemma}[cf. {\cite[1.5]{LS96}}]\label{cnbn}Let $G$ be a connected, reductive algebraic group over an algebraically closed field $k$ of characteristic $p$ %%%% 
and $Q$ a unipotent $G$-group. Suppose $H$ is a closed, connected, reductive subgroup of the semidirect product $GQ$ and denote by $\bar H$ the subgroup of $G$ given by the image of $H$ under the quotient map $\pi:GQ\to G$. Assume $\bar H=G$.

Then as abstract groups $H(k)$ is a complement to $Q(k)$ in $G(k)Q(k)$; and either (a) $H$ is a complement to $Q$ in $GQ$; or (b) \begin{enumerate}\item $p=2$;
\item $SO_{2n+1}$ is a component of the semisimple group $H/Z(H)^\circ$;
\item the image of this component in $G/Z(G)^\circ$ is isomorphic to $Sp_{2n}$; and
\item the natural module for $Sp_{2n}$ appears in a filtration of $Q$ by $G$-modules.\end{enumerate} In case (b), if $G$ and $Q$ and the action map $G\times Q\to Q$ are defined over the prime field $k_0$ of $k$, then $H$ corresponds to a cocycle $\gamma\in Z^1(G,Q^{[1]})$ such that $[\gamma]$ has no preimage in $H^1(G,Q)$ under the inclusion $H^1(G,Q)\to H^1(G,Q^{[1]})$.

Thus there is a bijection between the set of conjugacy classes of closed, connected, reductive  subgroups $H$ of $GQ$ with $\bar H=G$ and the set $H^1(G,Q^{[1]})$.
\end{lemma}
\begin{proof}We first tackle the case where $H$ (and therefore $G$) is semisimple. Then $H=H_1\dots H_n$. Let $G_i:=\pi(H_i)$. Clearly, $\prod G_i=G$.

Clearly if $H$ is a complement to $Q$ in $GQ$ then it is an abstract complement too. 

Assume $H$ is not a complement to $Q$. We still have $\bar HQ=HQ$ and so by Definition \ref{compdef} we must have either $H(k)\cap Q(k)\neq \{1\}$ or $\Lie(H)\cap \Lie(Q)\neq 0$. Since $H(k)\cap Q(k)\leq Z(H(k))\cap Q(k)$ the elements in the intersection are both semisimple and unipotent, hence just the identity.  It now follows that $H(k)$ is an abstract complement to $Q(k)$ in $G(k)Q(k)$. 

Now as $\Lie(H)\cap \Lie(Q)\neq 0$, we see that $\Lie(H)$ contains a nilpotent ideal $I=\Lie H\cap \Lie Q$, and the map of abstract groups $H(k)\to G(k)$ is an isomorphism. It follows that $p=2$, $H_i$ is of type $B_{n_i}$ for some $i$ and $n_i$ and $I$ is generated by the short root spaces (see \cite[Ch. 0]{Hum95}). Relabel the $H_i$ (and $G_i$) so that $H=H_1\dots H_mH_{m+1}\dots H_n$  with $\Lie H_i\cap I\neq 0$ if $1\leq i\leq m$ and $\Lie H_i\cap I=0$ otherwise; so $H_i$ is of type $B_{n_i}$ for $1\leq i\leq m$.

Since for each $1\leq i\leq m$, the image of $\Lie(H_i)$ is an ideal in $\Lie(G_i)$, we must have $G_i$ of type $C_{n_i}$. Now $H_i$ must be adjoint, or else $I$ contains the non-nilpotent subalgebra $Z(\Lie(H_i))$, since it contains the Lie algebra of a short root $SL_2$-subgroup. So $H_i=SO_{2n_i+1}$ and $G_i=Sp_{2n_i}$, with the map $\pi:H_i\to G_i$ being the well-known special isogeny. Define $\gamma:H\to GQ$ by $\gamma(h)=\pi(h)^{-1}h$. Then since $\gamma(h)\in \ker\pi$ we have in fact $\gamma:H\to Q$, with $\pi(h)=h\gamma(h)^{-1}$.

As $H$ is semisimple, $\Lie H_i\cap \Lie H_j=0$ for $i\neq j$ and as $H_i(k)\cap H_j(k)\leq Z(H_i)(k)=\{1\}$ for $1\leq i\leq m$ (as $H_i$ is adjoint), we have $H\cong H_1\times \dots\times H_m\times (H_{m+1}H_{m+2}\dots H_n)$. Now as $\Lie G_i\cap \Lie G_j=0$ for $i\neq j$ and that $\pi$ is an isomorphism of abstract groups, we must have $G\cong G_1\times\dots\times G_m\times (G_{m+1}\dots G_n)$ similarly. Let $J=\prod_{i=m+1}^nH_i$ and let $K=\prod_{i=m+1}^nH_i$. As the map $\pi|_J$ is an isomorphism of abstract groups $J\to K$ and induces an isomorphism on Lie algebras, $\pi|_J$ is an isomorphism by \cite[5.3.3]{Spr98}.

Now we can define a map $\sigma:G\to H$. Let $\sigma_i$ be the other well-known special isogeny $Sp_{2n_i}\to SO_{2n_i+1}$ for $1\leq i\leq m$, and $\tau:K\to J$ the inverse of $\pi|_J$. Then we define $\sigma:G\to H$ by setting $\sigma|_{H_i}=\sigma_i$ for $1\leq i\leq n$ and $\sigma|_K=\tau$.

Then the composition of maps $G\stackrel{\sigma}{\to} H\stackrel{\pi}{\to}G$ induces the first Frobenius endomorphism $F:H_i\to H_i$ for $1\leq i\leq m$, and an isomorphism on $K$. So under this morphism we have $g\mapsto h\mapsto h\gamma(h)^{-1}=F(g)$ for $g\in H_i$. Defining $\delta=\gamma\circ\sigma$ we have $h=F(g)\delta(g)$. Now one checks that $\delta$ defines a cocycle in $Z^1(G,Q^{[1]})$: for $g_1,g_2\in G_i$ with $1\leq i\leq m$, we have $F(g_1)\delta(g_1)F(g_2)\delta(g_2)=F(g_1)F(g_2)\delta(g_1g_2)$ by the fact that $\sigma_i$ is a homomorphism, so that $\delta(g_1g_2)=\delta(g_1)^{F(g_2)}\delta(g_2)$, showing that $\delta|_{G_i}$ defines a cocycle in $H^1(G_i,Q^{[1]})$. For $g_1,g_2\in K$, $\delta|_K$ defines a cocycle in $Z^1(K,Q)$, so that under the injection $Z^1(K,Q)\to Z^1(K,Q^{[1]})$, $\delta|_K$ defines a cocycle in $Z^1(K,Q^{[1]})$. Hence $\delta$ defines a cocycle in $Z^1(G,Q^{[1]})$.

It is clear that if a group $H^i$ gives rise to $\delta_i$ for $i\in\{1,2\}$ then $H^1=H^2\iff \delta_1=\delta_2$. Moreover if $H^1$ is conjugate to $H^2$ it is easy to check that $\delta_1\sim\delta_2$ in $Z^1(G,Q^{[1]})$.

Certainly $[\delta]$ can have no pre-image $[\gamma]$ in $H^1(G,Q)$: if $H'$ were the complement corresponding to $\gamma$ we would have that $H'\cong G$; and then $H=F(H')\cong G$---a contradiction. This proves (b).

Part (iv) follows from (b) noting that we must be in the exceptional case of Theorem \ref{rationalstab}.

The last statement is now clear: a cocycle class in $H^1(G,Q^{[1]})$ is either in the image of $H^1(G,Q)$ or it is not. This concludes the case where $H$ is semisimple.

Assume $H$ is a closed, connected, reductive subgroup of $GQ$. Since $Z(H)^\circ$ is a torus we may replace $H$ with a conjugate to have $Z(H)^\circ=Z(G)^\circ$. Thus each conjugacy class of subgroups considered has a representative $H$ with $Z:=Z(H)^\circ=Z(G)^\circ$. Let $H_1$ and $H_2$ be two such. Then $H_1$ is $Q$-conjugate to $H_2$ if and only if $H_1/Z$ is $Q^{Z}$-conjugate to $H_2/Z$ in the semidirect product $(G/Z)Q^{Z}$.

Thus the $Q$-conjugacy classes of reductive subgroups $H$ with $\bar H=G$ is in bijection with the $Q^Z$-conjugacy classes of semisimple subgroups $K$ of $(G/Z)Q^{Z}$ with $\bar K=G/Z$, where $\bar K$ denotes the image of $K$ under the quotient morphism $(G/Z)Q^Z\to G/Z$. We have shown already that this is in bijection with $H^1(G/Z,(Q^{Z})^{[1]})$. Now by \ref{lhs} we have $H^1(G/Z,(Q^{Z})^{[1]})\cong H^1(G,Q^{[1]})$, since $H^1(Z,Q^{[1]})=\{1\}$.
%Since $H(k)\cap Q(k)=1$ and $H(k)Q(k)=\bar H(k)Q(k)$, it follows that $H(k)$ is a complement to $Q(k)$ in $\bar H(k)Q(k)$. Thus $H(k)$ corresponds to some (non-rational) $1$-cocycle $\gamma:\bar H(k)\to Q(k)$. Taking a filtration of $Q$ by $H$-modules (which exists by \cite[3.3.5]{SteRes}), it follows that there is some composition factor $V$ with $H^1(H(k),V(k))\neq 0$. But $H^1(\bar H(k),V(k))\cong H^1_\mathrm{gen}(H,V)$ and this is isomorphic to $H^1(\bar H,V^{[1]})$ by \ref{cnrep}; by \cite[II.12.2, Remark]{Jan03} we have $H^1(\bar H,V)\cong H^1(\bar H,V^{[1]})$ unless $V$ is the natural module for $\bar H$, and $H^1(\bar H,V^{[1]})=k$.
\end{proof}

We recall Serre's notion of $G$-complete reducibility from \cite{Ser98}. A subgroup $H$ of $G$ is said to be $G$-completely reducible (or $G$-cr) if whenever $H$ is contained in a parabolic subgroup $P$ of $G$, it is contained in some Levi subgroup of that parabolic.

Using the above theorem we can show that a closed reductive subgroup $H$ of $G$ is $G$-cr if and only if whenever $H$ is in a parabolic subgroup $P$ of $G$, each parabolic subgroup $P_1$ of $H$ is in some Levi subgroup of that parabolic; in other words

\begin{corollary}\label{gcr}Let $H$ be a closed, connected, reductive subgroup of $G$ contained in a parabolic $P=R_u(P)L$ of $G$ and let $P_1$ be any parabolic  subgroup of $H$. Then $H$ is $G$-conjugate to a subgroup of $L$ if and only if $P_1$ is as well.

Moreover, if $U$ is the unipotent radical of a Borel subgroup of $H$ then $H$ is $G$-conjugate to a subgroup of $L$ if and only if $U$ is as well.
\end{corollary}
\begin{proof} One direction is trivial, so assume that $P_1$ is $G$-conjugate to a subgroup of $L$. 

Since $H$ is reductive, we have as usual that if $\bar H:=\pi(H)$ denotes the projection of $H$ to $L$ then $HQ=\bar HQ$ and we are in the situation of \ref{cnbn}. As $k$ is algebraically closed, it is perfect, hence $L$, $Q$ and the conjugation action $L\times Q\to Q$ are defined over $k_0$; so we have that $H$ corresponds to some cocycle $\gamma\in Z^1(\bar H,R_u(P)^{[1]})$. Now consider $\bar P_1:=\pi(P_1)\leq \bar H$. And let $\beta$ denote the restriction $\gamma\downarrow^{\bar H}_{\bar P_1}$. Then $\beta\in Z^1(\bar P_1,R_u(P)^{[1]})$.
 
The hypothesis that $P_1$ is $G$-conjugate to $L$ implies that $P_1$ is $R_u(P)$-conjugate to $L$ by \cite[5.9(ii)]{BMRT09}. Thus $\beta$ must be in the trivial cocycle class in $H^1(\bar P_1,R_u(P)^{[1]})$. But by Theorem \ref{alggrp}, $H^1(\bar P_1,R_u(P))\cong H^1(\bar H,R_u(P)^{[1]})$ and so $\gamma$ is in the trivial cocycle class in $H^1(\bar H,R_u(P)^{[1]})$. Thus $H$ is $R_u(P)$-conjugate to a subgroup of $L$ and so clearly it is $G$-conjugate to a subgroup of $L$.

Similarly, for the last part of the corollary, if $P_1=B$ is a Borel subgroup of $H$, then we have $H^1(\bar H,R_u(P)^{[1]})\cong H^1(\bar B,R_u(P)^{[1]})$ and by \ref{cartanalg} the restriction map to $H^1(\bar U,R_u(P)^{[1]})^T$ is injective. Suppose $U=R_uB$ is $R_u(P)$-conjugate to a subgroup of $L$. Then if $H$ corresponds to $\gamma$ and $\beta$ denotes the restriction of $\gamma$ to $\bar U$, then the hypothesis implies that $\beta$ is in the trivial cocycle class. But by the injectivity of the restriction map, $\gamma$ must have been in the trivial cocycle class and so $H$ is $G$-conjugate to $L$.
\end{proof}

\begin{corollary}\label{finitegcr}Let $H$ be a closed, connected, reductive subgroup of $G$ contained in a parabolic $P=R_u(P)L$ of $G$. Then there exists a finite subgroup $H(q)$ of $H$ such that $H(q)$ is conjugate to a subgroup of $L$ if and only if $H$ is also.\end{corollary}
\begin{proof} This works the same as the previous corollary, using the following fact: As $L$, $R_u(P)$, and the conjugation action of $L$ on $R_u(P)$ are defined over the prime field $k_0$, we can invoke \ref{rationalstab} to obtain an isomorphism $H^1(H(q),R_u(P))\cong H^1(H,R_u(P))$.\end{proof}

\begin{remark}%%%%
Note that in above theorem one can use Remark \ref{explicit} to get an explicit value of $q$.\end{remark} 

{\bf Acknowledgements.} The author would like to thank Martin Liebeck, Steve Donkin, George McNinch and Rapha\"el Rouquier for helpful conversations and suggestions. He would also like to thank the anonymous referee for several helpful suggestions for improvement.
{\footnotesize
\bibliographystyle{amsalpha}
\bibliography{stewart}}
\end{document}